\documentclass[twoside,leqno,10pt, A4]{amsart}
\usepackage{amsfonts}
\usepackage{amsmath}
\usepackage{amscd}
\usepackage{amssymb}
\usepackage{amsthm}
\usepackage{amsrefs}
\usepackage{latexsym}
\usepackage{mathrsfs}
\usepackage{bbm}
\usepackage{amscd}
\usepackage{amssymb}
\usepackage{amsthm}
\usepackage{amsrefs}
\usepackage{latexsym}
\usepackage{mathrsfs}
\usepackage{bbm}
\usepackage{enumerate}
\usepackage{graphicx}
\usepackage{color}
\begin{document}

\newtheorem{theorem}[subsection]{Theorem}
\newtheorem{proposition}[subsection]{Proposition}
\newtheorem{lemma}[subsection]{Lemma}
\newtheorem{corollary}[subsection]{Corollary}
\newtheorem{conjecture}[subsection]{Conjecture}
\newtheorem{prop}[subsection]{Proposition}
\numberwithin{equation}{section}
\newcommand{\mr}{\ensuremath{\mathbb R}}
\newcommand{\mc}{\ensuremath{\mathbb C}}
\newcommand{\dif}{\mathrm{d}}
\newcommand{\intz}{\mathbb{Z}}
\newcommand{\ratq}{\mathbb{Q}}
\newcommand{\natn}{\mathbb{N}}
\newcommand{\comc}{\mathbb{C}}
\newcommand{\rear}{\mathbb{R}}
\newcommand{\prip}{\mathbb{P}}
\newcommand{\uph}{\mathbb{H}}
\newcommand{\fief}{\mathbb{F}}
\newcommand{\majorarc}{\mathfrak{M}}
\newcommand{\minorarc}{\mathfrak{m}}
\newcommand{\sings}{\mathfrak{S}}
\newcommand{\fA}{\ensuremath{\mathfrak A}}
\newcommand{\mn}{\ensuremath{\mathbb N}}
\newcommand{\mq}{\ensuremath{\mathbb Q}}
\newcommand{\half}{\tfrac{1}{2}}
\newcommand{\f}{f\times \chi}
\newcommand{\summ}{\mathop{{\sum}^{\star}}}
\newcommand{\chiq}{\chi \bmod q}
\newcommand{\chidb}{\chi \bmod db}
\newcommand{\chid}{\chi \bmod d}
\newcommand{\sym}{\text{sym}^2}
\newcommand{\hhalf}{\tfrac{1}{2}}
\newcommand{\sumstar}{\sideset{}{^*}\sum}
\newcommand{\sumprime}{\sideset{}{'}\sum}
\newcommand{\sumprimeprime}{\sideset{}{''}\sum}
\newcommand{\sumflat}{\sideset{}{^\flat}\sum}
\newcommand{\shortmod}{\ensuremath{\negthickspace \negthickspace \negthickspace \pmod}}
\newcommand{\V}{V\left(\frac{nm}{q^2}\right)}
\newcommand{\sumi}{\mathop{{\sum}^{\dagger}}}
\newcommand{\mz}{\ensuremath{\mathbb Z}}
\newcommand{\leg}[2]{\left(\frac{#1}{#2}\right)}
\newcommand{\muK}{\mu_{\omega}}
\newcommand{\thalf}{\tfrac12}
\newcommand{\lp}{\left(}
\newcommand{\rp}{\right)}
\newcommand{\Lam}{\Lambda_{[i]}}
\newcommand{\lam}{\lambda}
\newcommand{\af}{\mathfrak{a}}
\newcommand{\sw}{S_{[i]}(X,Y;\Phi,\Psi)}
\newcommand{\lz}{\left(}
\newcommand{\pz}{\right)}
\newcommand{\bfrac}[2]{\lz\frac{#1}{#2}\pz}
\newcommand{\odd}{\mathrm{\ primary}}
\newcommand{\even}{\text{ even}}
\newcommand{\res}{\mathrm{Res}}

\theoremstyle{plain}
\newtheorem{conj}{Conjecture}
\newtheorem{remark}[subsection]{Remark}

\makeatletter
\def\widebreve{\mathpalette\wide@breve}
\def\wide@breve#1#2{\sbox\z@{$#1#2$}%
     \mathop{\vbox{\m@th\ialign{##\crcr
\kern0.08em\brevefill#1{0.8\wd\z@}\crcr\noalign{\nointerlineskip}%
                    $\hss#1#2\hss$\crcr}}}\limits}
\def\brevefill#1#2{$\m@th\sbox\tw@{$#1($}%
  \hss\resizebox{#2}{\wd\tw@}{\rotatebox[origin=c]{90}{\upshape(}}\hss$}
\makeatletter

\title[First moment of central values of quadratic Hecke $L$-functions in the Gaussian field]{First moment of central values of quadratic Hecke $L$-functions in the Gaussian field}

\author{Peng Gao and Liangyi Zhao}

\begin{abstract}
 We evaluate the smoothed first moment of central values of a family of qudratic Hecke $L$-functions in the Gaussian field using the method of double Dirichlet series.  The asymptotic formula we obtain has an error term of size $O(X^{1/4+\varepsilon})$ under the generalized Riemann hypothesis.  The same approach also allows us to obtain asymptotic formulas for all $X$, $Y$ for a smoothed double character sum involving $\sum_{N(m) \leq X, N(n)\leq Y}\leg {m}{n}$, where $\leg {\cdot}{n}$ denotes the quadratic symbol in the Gaussian field.
\end{abstract}

\maketitle



\section{Introduction}
\label{sec 1}

  Moments of central values of families of $L$-functions have been intensively studied in the literature due to their broad arithmetic applications. For the family of quadratic Dirichlet $L$-functions, M. Jutila \cite{Jutila} obtained an asymptotic formula for the first moment with the main term of size $X \log X$ and an error term of size $O(X^{3/4+\varepsilon})$ for any $\varepsilon>0$. The same error term was later obtained by A. I. Vinogradov and L. A. Takhtadzhyan \cite{ViTa} and was improved to $O(X^{19/32 + \varepsilon})$ for any $\varepsilon>0$  by D. Goldfeld and J. Hoffstein \cite{DoHo}. In fact, an error term of size $O(X^{1/2 + \varepsilon})$ for the smoothed first moment is essentially implicit in \cite{DoHo} and this result was established by M. P. Young \cite{Young1} through a different approach using an recursive argument.
It was conjectured in \cite{DoHo} that the optimal error term should be $O(X^{1/4+\varepsilon})$ for the first moment studied therein. This has also been observed in a numerical study conducted by M. W. Alderson and M. O. Rubinstein \cite{AR12}. \newline

  The result of Goldfeld and Hoffstein was achieved using the method of double Dirichlet series via the usage of Eisenstein series of metaplectic type.  This appraoch gives the desired analytical property of the double Dirichlet series involved.  Later, A. Diaconu, D. Goldfeld and J. Hoffstein \cite{DGH} discussed the many advantages in viewing multiple Dirichlet series as functions of several complex variables.  This point of view was then applied in \cite{DGH} to study the third moment of central values of the family of quadratic Dirichlet $L$-functions. \newline

  Other than moments of $L$-functions, the method of multiple Dirichlet series can also be applied to study many related results. Take, for example, the following sum over quadratic Dirichlet characters given by
\begin{align*}
  S(X,Y)=\sum_{\substack {m \leq X \\ (m, 2)=1}}\sum_{\substack {n \leq Y \\ (n, 2)=1}} \leg {m}{n}_{\mz},
\end{align*}
  where $\left(\frac{m}{\cdot} \right)_{\mz}$ stands for the Kronecker symbol.
In \cite{CFS}, J. B. Conrey, D. W. Farmer and K. Soundararajan evaluated $S(X,Y)$ asymptotically by observing that an asymptotic formula of $S(X, Y)$ can be obtained via
the P\'olya-Vinogradov inequality if the sizes of $X$ and $Y$ are slightly different.  However, it is quite delicate to find such an asymptotic formula for $X$ and $Y$ of comparable size.  This difficulty was resolved in \cite{CFS} using a Poisson summation formula developed by Soundararajan \cite{sound1} so that a valid asymptotic formula for all $X, Y$ is established for $S(X, Y)$. \newline

In \cite{Cech}, M. \v Cech revisited the function $S(X,Y)$ from the vantage point of double Dirichlet series, applying the inverse Mellin transforms to express $S(X,Y)$ as integrals involving a double Dirichlet series. The corresponding double Dirichlet series appeared earlier in the work of V. Blomer \cite{Blomer11}, who established a subconvexity bound for it.  Using the analytical property of the double Dirichlet series obtained in \cite{Blomer11}, \v Cech was able to obtain an asymptotic formula for $S(X, Y)$ with a better error term. \newline

  We note that the method of Blomer also leads to the conjectured size of error term concerning the first moment of central values of the related quadratic family of Dirichlet $L$-functions under the generalized Riemann hypothesis (GRH).  In the present paper, we shall demonstrate a similar phenomenon in the setting of the Gaussian field.  We set $K=\mq(i)$ for the Gaussian field throughout the paper and $\mathcal{O}_K=\mz[i]$ for the ring of integers of $K$.  Let $N(n)$ stand for the norm of any $n \in K$ and $\zeta_K(s)$ the Dedekind zeta function of $K$.  It will be shown in Section~\ref{sec2.4} below that every ideal in $\mathcal{O}_K$ co-prime to $2$ has a unique generator congruent to $1$ modulo $(1+i)^3$ which is called primary.  Also, the quadratic symbol $\chi_m :=\left(\frac{m}{\cdot} \right)$ is defined in Section \ref{sec2.4}, which can be viewed as analogous to the Kronecker symbol in $K$. \newline

We fix a non-negative, smooth function $\Phi(x)$ compactly supported on ${\mr}_+$, where ${\mr}_+$ denotes the set of positive real numbers.  We are interested in the smoothed first moment of central values of the quadratic family of Hecke $L$-functions given by
\begin{align*}
    \sum_{m\odd} L \left( \frac 12, \chi_m \right) \Phi\left( \frac {N(m)}X \right).
\end{align*}

We apply the double Dirichlet series method in this paper to evaluate the above sum asymptotically.  For any function $f$, let $\hat f$ be the Mellin transform of $f$ given in \eqref{fMellin}.  Our result is as follows.
\begin{theorem}
\label{Thmfirstmomentatcentral}
	For all large $X$ and any $\varepsilon>0$, we have
\begin{align}
\label{FirstmomentSmoothed}
\begin{split}
	 \sum_{m\odd} L \left( \frac 12, \chi_m \right) \Phi\left( \frac {N(m)}X \right) =XP(\log X)+O\lz X^{\delta+\varepsilon}\pz ,
\end{split}
\end{align}
  where $P$ is a linear polynomial whose coefficients depend only on absolute constants and $\hat \Phi(1)$, $\hat \Phi'(1)$ and $\delta=1/2$.  If we assume the truth of GRH, then we can take $\delta=1/4$.
\end{theorem}

  As our main focus is the error term here, we omit the explicit expression of $P$ by pointing out that one can obtain it by noting that the function $A(s,w)$ defined in \eqref{Aswexp} has a double pole at $s=1$, $w=1/2$. A result similar to Theorem \ref{Thmfirstmomentatcentral} has been proved by the first-named author \cite{Gao20} using Young's recursive method.  The error term obtained in \cite{Gao20} is $O\lz X^{1/2+\varepsilon}\pz$ unconditionally.  We also remark here that, similar to what has been pointed out in \cite{DoHo} concerning quadratic Dirichlet $L$-functions, it seems unlikely that any improvement of  the $O$-term in \eqref{FirstmomentSmoothed} can be attained without some substantial enlargement of the zero-free region for the Dedekind zeta function.  Our proof of Theorem \ref{FirstmomentSmoothed} clearly illustrates this as one sees that our derivation of the error term largely depends on the Dedekind zeta function appearing in the denominator of \eqref{Firstmomentintrep}. \newline

   Our approach  in the proof of Theorem \ref{Thmfirstmomentatcentral} also allows us to study a sum analogous to $S(X, Y)$ involving with Hecke characters. For this, we fix another non-negative, smooth function $\Psi(x)$ that is compactly supported on ${\mr}_+$ and we define
\begin{align*}
    S_{[i]}(X,Y;\Phi,\Psi)=\sum_{m\odd}  \sum_{n\odd}\leg mn \Phi\left( \frac {N(m)}X \right) \Psi \left( \frac {N(n)}Y \right).
\end{align*}

   Our next result evaluates $S_{[i]}(X,Y;\Phi,\Psi)$ asymptotically.
\begin{theorem}
\label{MainTheoremSmoothed}
	For all large $X,Y,$ we have
\begin{equation*}
	 S_{[i]}(X,Y;\Phi,\Psi)=\frac {\pi^2}{48\zeta_K(2)}\cdot X^{3/2}\cdot D\lz\frac YX;\Phi,\Psi\pz+O_{\varepsilon}(XY^\delta+YX^\delta),
\end{equation*}
	where
\begin{align*}
\begin{split}
D(\alpha;\Phi,\Psi)=&\frac{\hat\Phi(1)\hat\Psi\bfrac12\alpha^{1/2}+\hat\Psi(1)\hat\Phi\bfrac12\alpha}{2}\\
 & \hspace*{1cm} +\frac 1{2\pi i} \int\limits_{(3/4)}\alpha^s\hat\Phi\lz\frac32-s\pz\hat\Psi(s)\pi^{1-3s}2^{2-2s}\frac {\Gamma(1-s)\Gamma(2s-1)}{\Gamma(2-2s)\Gamma(s)}\zeta_K(2s-1) \dif s.
\end{split}
\end{align*}
  and where $\delta=\varepsilon$ for any $\varepsilon>0$. If we assume the truth of GRH, then we can take $\delta=-1/4+\varepsilon$.
\end{theorem}

   We note that sums similar to $S_{[i]}(X,Y;\Phi,\Psi)$ have already been evaluated by the authors in \cites{G&Zhao2019, G&Zhao2022-3}, but asymptotically only for certain ranges of $X$ and $Y$.  Theorem~\ref{MainTheoremSmoothed} gives a valid asymptotical formula for all $X$ and $Y$.

\section{Preliminaries}
\label{sec 2}

To prepare for the proofs of the theorems, we first include some auxiliary results needed in the paper.

\subsection{Quadratic residue symbol and quadratic Hecke character}
\label{sec2.4}
   Recall that $K=\mq(i)$ and it is well-known that $K$ has class number one.  We write $U_K=\{ \pm 1, \pm i \}$ and $D_{K}=-4$ for the group of units in $\mathcal{O}_K$ and the discriminant of $K$, respectively. We say an element $\varpi \in \mathcal{O}_K$ is prime if $(\varpi)$ is a prime ideal and an element $n \in \mathcal{O}_K$ is square-free if no prime squares divides $n$. We shall henceforth reserve the letter $\varpi$ for a prime in $\mathcal O_K$. Note that $n$ is square-free if and only if $\mu_{[i]}(n) \neq 0$, where $\mu_{[i]}$ is the M\"obius function on $\mathcal{O}_K$.  Note that $(1+i)$ is the only prime ideal in $\mathcal{O}_K$ that lies above the integral ideal $(2)$ in $\mz$ and we say an element $n \in \mathcal{O}_K$ is odd if $(n,1+i)=1$. \newline

   For $q \in \mathcal{O}_K$, let $\left (\mathcal{O}_K / (q) \right )^*$ denote the group of reduced residue classes modulo $q$, i.e., the multiplicative group of invertible elements in $\mathcal{O}_K / (q)$. Note that $\left (\mathcal{O}_K / ((1+i)^3) \right )^*$ is isomorphic to the cyclic group of order four generated by $i$. This implies that every ideal in $\mathcal{O}_K$ co-prime to $2$ has a unique generator congruent to $1$ modulo $(1+i)^3$.  This generator is called primary.  Furthermore, $\left (\mathcal{O}_K / (4) \right )^*$ is isomorphic to the direct product of a cyclic group of order $4$ and a cyclic group of order $2$.  More precisely, we have
\begin{align}
\label{resclassmod4}
  \left (\mathcal{O}_K / (4) \right )^* \cong  <i> \times <1+ 2(1+i)>.
\end{align}
  Note that the class $1 \pmod {(1+i)^3}$ gives rise to the two classes $1+2(1+i)k, k \in \{0,1 \}$ modulo $4$. It follows from this that an element $n=a+bi \in \mathcal{O}_K$ with $a, b \in \mz$ is primary if and only if $a \equiv 1 \pmod{4}, b \equiv
0 \pmod{4}$ or $a \equiv 3 \pmod{4}, b \equiv 2 \pmod{4}$. This result can be found in Lemma 6 on \cite[p. 121]{I&R}.  We shall refer to these two cases above as $n$ is of type $1$ and type $2$, respectively. \newline

   Let $0 \neq q \in \mathcal{O}_K$ and $\chi$ be a homomorphism:
\begin{align*}
  \chi: \left (\mathcal{O}_K / (q) \right )^*  \rightarrow S^1 :=\{ z \in \mc :  |z|=1 \}.
\end{align*}
     Following the nomenclature of \cite[Section 3.8]{iwakow}, we shall refer $\chi$ as a Dirichlet character modulo $q$. We say that such a Dirichlet character $\chi$ modulo $q$ is primitive if it does not factor through $\left (\mathcal{O}_K / (q') \right )^*$ for any divisor $q'$ of $q$ with $N(q')<N(q)$. \newline

   For any odd $n \in \mathcal{O}_{K}$, we write the symbol  $\leg {\cdot}{n}$ for the quadratic residue symbol modulo $n$ in $K$. For a prime $\varpi \in \mathcal{O}_{K}$
with $N(\varpi) \neq 2$, the quadratic symbol is defined for $a \in
\mathcal{O}_{K}$, $(a, \varpi)=1$ by $\leg{a}{\varpi} \equiv
a^{(N(\varpi)-1)/2} \pmod{\varpi}$, with $\leg{a}{\varpi} \in \{
\pm 1 \}$.  If $\varpi | a$, we define
$\leg{a}{\varpi} =0$.  Then the quadratic symbol is extended
to any composite odd $n$  multiplicatively. We further define $\leg {\cdot}{c}=1$ for $c \in U_K$. Recall that $\chi_c$ stands for the symbol $\leg {c}{\cdot}$, where we define $\leg {c}{n}=0$ when $1+i|n$. \newline

  The following quadratic reciprocity law (see \cite[(2.1)]{G&Zhao4}) holds for two co-prime primary elements $m$, $n \in \mathcal{O}_{K}$:
\begin{align}
\label{quadrec}
 \leg{m}{n} = \leg{n}{m}.
\end{align}
Moreover, we deduce from Lemma 8.2.1 and Theorem 8.2.4 in \cite{BEW} that the following supplementary laws hold for primary $n=a+bi$ with $a, b \in \mz$:
\begin{align}
\label{supprule}
  \leg {i}{n}=(-1)^{(1-a)/2} \qquad \mbox{and} \qquad  \hspace{0.1in} \leg {1+i}{n}=(-1)^{(a-b-1-b^2)/4}.
\end{align}

  Note that the quadratic symbol $\leg {\cdot}{n}$ defined above is a Dirichlet character modulo $n$. It follows from \eqref{supprule} that the quadratic symbol $\psi_i:=\leg {i}{\cdot}$ defines a primitive Dirichlet character modulo $4$.  Also,  \eqref{supprule} implies that  the quadratic symbol $\psi_{1+i}:=\leg {1+i}{\cdot}$ defines a primitive Dirichlet character modulo $(1+i)^5$. This can be inferred by noting that
\begin{align*}
  \left (\mathcal{O}_K / ((1+i)^5) \right )^* \cong  <i> \times <1+ 2(1+i)>\times <5>.
\end{align*}
As $5 \equiv 1 \pmod 4$ and $\psi_{1+i}(5)=-1$, we get that $\psi_{1+i}$ must be primitive. \newline

Furthermore, we observe that there exists a primitive quadratic Dirichlet character $\psi_2$ modulo $2$ since $\left (\mathcal{O}_K / (2) \right )^*$ is isomorphic to the cyclic group of order two generated by $i$.  Then we have $\psi_2(n)=-1$ for $n \equiv i \pmod 2$. \newline

   For any $l \in \mz$ with $4 | l$, we define a unitary character $\chi_{\infty}$ from $\mc^*$ to $S^1$ by:
\begin{align*}
  \chi_{\infty}(z)=\leg {z}{|z|}^l.
\end{align*}
  The integer $l$ is called the frequency of $\chi_{\infty}$.

  Now, for a given Dirichlet character $\chi$ modulo $q$ and a unitary character $\chi_{\infty}$, we can define a Hecke character $\psi$ modulo $q$ (see \cite[Section 3.8]{iwakow}) on the group of fractional ideals $I_K$ in $K$, such that for any $(\alpha) \in I_K$,
\begin{align*}
  \psi((\alpha))=\chi(\alpha)\chi_{\infty}(\alpha).
\end{align*}
If $l=0$, we say that $\psi$ is a Hecke character modulo $q$ of trivial infinite type. In this case, we may regard $\psi$ as defined on $\mathcal{O}_K$ instead of on $I_K$, setting $\psi(\alpha)=\psi((\alpha))$ for any $\alpha \in \mathcal{O}_K$. We may also write $\chi$ for $\psi$ as well, since we have $\psi(\alpha)=\chi(\alpha)$ for any $\alpha \in \mathcal{O}_K$. We then say such a Hecke character $\chi$ is primitive modulo $q$ if $\chi$ is a primitive Dirichlet character. Likewise, we say that  $\chi$ is induced by a primitive Hecke character $\chi'$ modulo $q'$ if $\chi(\alpha)=\chi'(\alpha)$ for all
$(\alpha, q')=1$. \newline

   We now define an abelian group $\text{CG}$ such that it is generated by three primitive quadratic Hecke characters of trivial infinite type with corresponding moduli dividing $(1+i)^5$.  More precisely,
\begin{align*}
  \text{CG}=\{ \psi_j : j=1, i, 1+i, i(1+i) \},
\end{align*}
and the commutative binary operation on $\text{CG}$ is given by $\psi_i \cdot \psi_{i(1+i)}=\psi_{1+i}$, $\psi_{1+i} \cdot \psi_{i(1+i)}=\psi_i$ and $\psi_j \cdot \psi_j=\psi_1$ for any $j$. As we shall only evaluate the related characters at primary elements in $\mathcal{O}_K$, the definition of such a product is therefore justified. \newline

  We note that the product of $\chi_c$ for any primary $c$ with any $\psi_j \in \text{CG}$ gives rise to a Hecke character of trivial infinite type. To determine the primitive Hecke character that induces such a product, we observe that every primary $c$ can be written uniquely as
\begin{align*}
 c=c_1c_2, \quad \text{$c_1, c_2$ primary and $c_1$ square-free}.
\end{align*}
The above decomposition allows us to conclude that if $c_1$ is of type 1, then $\chi_c \cdot \psi_j$ for $j \in \{1, i, 1+i, i(1+i)\}$ is induced by the primitive Hecke character $\leg {\cdot}{c_1}\cdot \psi_j$ with modulus $c_1, 4c_1, (1+i)^5c_1$ and $(1+i)^5c_1$ for $j=1, i, 1+i$ and $i(1+i)$, respectively. This is because that $\leg {\cdot }{c_1}$ is trivial on $U_K$ by \eqref{supprule}. Similarly, if $c_1$ is of type 2,  $\chi_c \cdot \psi_j$ for $j \in \{1, i, 1+i, i(1+i) \}$ is induced by the primitive Hecke character $\psi_j \cdot \psi_2 \cdot \leg {\cdot}{c_1}$ with modulus $2c_1, 4c_1, (1+i)^5c_1$ and $(1+i)^5c_1$ for $j=1, i, 1+i$ and $i(1+i)$, respectively. We make the convention that for any primary $n \in \mathcal O_K$, we shall use $\chi_n\cdot \psi_j$ for any $\psi_j \in \text{CG}$ to denote the corresponding primitive Hecke character $\chi$ that induces it throughout the paper.

\subsection{The approximate functional equation}
\label{sect: apprfcneqn}

  For any complex number $z$, we define\begin{align*}
 \widetilde{e}(z) =\exp \left( 2\pi i  \left( \frac {z}{2i} - \frac {\bar{z}}{2i} \right) \right) .
\end{align*}
  Let $\chi$ be a primitive quadratic Hecke character modulo $q$ of trivial infinite type of $K$ or a primitive Dirichlet character modulo $q$. We define the Gauss sum $g(\chi)$ associated to $\chi$  by
\begin{align*}
 g(\chi) = \sum_{x \bmod{q}} \chi(x) \widetilde{e}\leg{x}{q}.
\end{align*}

  Recall that $\zeta_K(s)$ denotes the Dedekind zeta function of $K$.  We also define $L_c(s, \chi)$, and also $\zeta_c(s)$, to be given by the Euler product defining $L(s, \chi)$ or $\zeta_K(s)$ but omitting those primes dividing $c$.  A well-known result of E. Hecke shows that $L(s, \chi)$ has an analytic continuation to the whole complex plane and satisfies the functional equation (see \cite[Theorem 3.8]{iwakow})
\begin{align}
\label{fneqn}
  \Lambda(s, \chi) = W(\chi)\Lambda(1-s, \chi),
\end{align}
  where
\begin{align}
\label{Lambda}
  \Lambda(s, \chi) = (|D_K|N(q))^{s/2}(2\pi)^{-s}\Gamma(s)L(s, \chi),
\end{align}
   and
\begin{align*}
  W(\chi) = g(\chi)(N(q))^{-1/2}.
\end{align*}

  As in the case of the classical quadratic Gauss sum over $\mz$, one expects that the Gauss sum  $g(\chi)$ associated to any primitive quadratic Hecke character $\chi$ modulo $q$ equals $N(q)^{1/2}$. Our next result confirms this in our situation.
\begin{lemma}
\label{Gausssumevaluation}
 With notations as above, for any primary, square-free $n \in \mathcal{O}_K$, let $\chi=\chi_n \cdot \psi_j$ with $\psi_j \in \text{CG}$ and $q$ be the modulus of $\chi$.  We have
\begin{align*}
   g(\chi)=N(q)^{1/2}.
\end{align*}
\end{lemma}
\begin{proof}
  First note that the assertion of the lemma is shown to be true for the case $j=1+i, i(1+i)$ in \cite[Lemma 2.2]{Gao2}. We may thus assume that $j=1, i$ in what follows and we note that in this case the Hecke character $\chi$ can be written as $\chi=\chi_c$ or $\chi_{ic}$ for a primary $c \in \mathcal O_K$. We recall from the display above \cite[(2.2)]{G&Zhao3} that for a primary $c \in \mathcal O_K$,
\begin{align}
\label{gn}
   g \Big(\leg {\cdot}{c} \Big )=\leg {i}{c}N(c)^{1/2}.
\end{align}
  It follows from this,  the definition of $\leg {c}{\cdot}$ and \eqref{supprule} that we have
$g(\leg {c}{\cdot})=1$ if $c$ is of type $1$. Moreover, the Chinese remainder theorem implies that $x = 4y +c z$ varies over the residue class modulo $4c$ when $y$ and $z$ vary over the residue class modulo $c$ and $4$, respectively. This and \eqref{gn} give that, when $c$ is of type $1$,
\begin{align*}
   g\Big(\leg {ic}{\cdot}\Big)=\leg {i}{c}\leg {4}{c}g\Big(\leg {i}{\cdot}\Big ) g\Big(\leg {\cdot}{c}\Big )=g\Big(\leg {i}{\cdot}\Big ) N(c)^{1/2}=N(2c)^{1/2},
\end{align*}
 where the last equality above follows from using the representation of $\left (\mathcal{O}_K / (4) \right )^*$ given in \eqref{resclassmod4} to see that $g(\leg {i}{\cdot} )=4$. \newline

On the other hand, if $c$ is of type 2, then the Chinese remainder theorem implies that $x = 2y +c z$ varies over the residue class modulo $2c$ when $y$ and $z$ vary over the residue class modulo $c$ and $2$, respectively.  As before, this, together with \eqref{gn}, implies that
\begin{align*}
   g\Big(\leg {c}{\cdot}\Big)=\psi_2(c)\leg {2}{c}g(\psi_{2}) g\Big(\leg {\cdot}{c}\Big )=\psi_2(c)\leg {2i}{c}g(\psi_{2}) N(c)^{1/2}.
\end{align*}
   Using that the representation of $\left (\mathcal{O}_K / (2) \right )^*$ can be chosen to consist of $1, i$, we see that $g(\psi_{2})=2$. Moreover, we have $\leg {2i}{c}=\leg {(1+i)^2}{c}=1$, and that $\psi_2(c)=1$ since we have $c \equiv 1 \pmod 2$. It follows that
\begin{align*}
   g\Big(\leg {c}{\cdot}\Big)=N(2c)^{1/2}.
\end{align*}

  Similarly, when $c$ is of type 2, we have
\begin{align*}
   g\Big(\leg {ic}{\cdot}\Big)=\leg {i}{c}\psi_2(c)\leg {4}{c}g\Big(\leg {i}{\cdot}\cdot \psi_2\Big ) g\Big(\leg {\cdot}{c} \Big )=g\Big(\leg {i}{\cdot}\cdot \psi_2 \Big ) N(c)^{1/2},
\end{align*}
  where we note that $\leg {i}{\cdot} \cdot\psi_2$ is a primitive Dirichlet character modulo $4$.  Using again the representation of $\left (\mathcal{O}_K / (4) \right )^*$ given in \eqref{resclassmod4}, we see via direct computation that $g(\leg {i}{\cdot}\cdot\psi_2 )=4$.
This completes the proof.
\end{proof}

We deduce from Lemma~\ref{Gausssumevaluation} that for any primary, square-free $c \in \mathcal{O}_K$ and any $\chi=\chi_c \cdot \psi_j$ with $\psi_j \in \text{CG }$, we have $W(\chi ) = 1$. We conclude from this, \eqref{fneqn} and \eqref{Lambda} that if such $\chi$ is of modulus $q$, then
\begin{align}
\label{fneqnL}
  L(1-s, \chi)=N(q)^{(2s-1)/2}\pi^{1-2s}\Gamma^{-1}(1-s)\Gamma(s)L(s, \chi).
\end{align}

\subsection{A mean value estimate for quadratic Hecke $L$-functions}
 In the proof of Theorem \ref{MainTheoremSmoothed}, we need the following lemma, a consequence of \cite[Corollary 1.4]{BGL}), which estimates from above the second moment of quadratic Hecke $L$-functions.
\begin{lemma}
\label{lem:2.3}
Suppose $s$ is a complex number with $\Re(s) \geq \frac{1}{2}$ and $|s-1|>\varepsilon$ for any $\varepsilon>0$. Then for any $\psi \in \text{CG}$,
\begin{align}
\label{L4est}
\sumstar_{\substack{(m,2)=1 \\ N(m) \leq X}} |L(s,\chi_{m}\psi)|^2
\ll (X|s|)^{1+\varepsilon}.
\end{align}
\end{lemma}

\section{Proof of Theorem \ref{Thmfirstmomentatcentral}}
\label{sec Poisson}

  We recall that the Mellin transform $\hat{f}$ for any function $f$ is defined to be
\begin{align}
\label{fMellin}
     \widehat{f}(s) =\int\limits^{\infty}_0f(t)t^s\frac {\dif t}{t}.
\end{align}

Repeated integration by parts yields that for the weight function $\Phi$ appearing on the left-hand side of \eqref{FirstmomentSmoothed} and any integer $E \geq 0$,
\begin{align} \label{h1bound}
 \hat \Phi(w) \ll  \frac{1}{(1+|w|)^{E}}.
\end{align}

Applying \eqref{h1bound} and the Mellin inversion shows that for $\Re(s)$ large enough,
\begin{align}
\label{MellinInversion}
\sum_{m\odd} L(s, \chi_m)\Phi \left( \frac {N(m)}X \right)=\frac1{2\pi i}\int\limits_{(\omega)}A(s,w)X^w\hat\Phi(w) \dif w,
\end{align}
  where we define $A(s,w)$, for $\Re(s)$ and $\Re(w)=\omega$ large enough, by the absolutely convergent double Dirichlet series
\begin{equation}
\label{Aswexp}
A(s,w)=\sum_{m\odd}\sum_{n\odd}\frac{\leg mn}{N(m)^wN(n)^s}.
\end{equation}

  We are thus led to study the analytical properties of $A(s,w)$. For this, it turns out that we need to extend our consideration to a larger class of functions.  In this vein, we define for two quadratic Hecke characters $\psi,\psi' \in \text{CG}$,
\begin{equation*}
 Z(s,w;\psi,\psi')=\zeta_2(2s+2w-1)\sum_{m\odd}\sum_{n\odd}\frac{\leg mn \psi(n)\psi'(m)}{N(m)^wN(n)^s}.
\end{equation*}

   We also introduce a $\mz$ module structure on $\text{CG}$ in the sense that we define for $\psi'' \in \text{CG}$ and $m, n \in \mz$,
\begin{equation*}
 Z(s,w;\psi,m\psi'+n\psi'')=mZ(s,w;\psi,\psi')+nZ(s,w;\psi,\psi'').
\end{equation*}
   We further define $Z(s,w)=Z(s,w;\psi_1,\psi_1)$ so that
\begin{equation}
\label{AZ}
A(s,w)=\frac {Z(s,w)}{\zeta_2(2s+2w-1)}.
\end{equation}

\subsection{Analytical Properties of $Z(s,w;\psi,\psi')$}

 Our next result describes the analytical properties of $Z(s,w;\psi,\psi')$.
\begin{proposition}
\label{ExtensionOfZ(s,w)}
	The function $Z(s,w;\psi,\psi')$ has a meromorphic continuation to the whole $\mc^2$ with a polar line $s+w=3/2$. There is an additional polar line at $s=1$ with residue $\res_{(1,w)}Z(s,w)=\pi \zeta_2(2w)/8$ if and only if $\psi=\psi_1$, and an additional polar line $w=1$ with residue $\res_{(s,1)}Z(s,w)=\pi \zeta_2(2s)/8$ if and only if $\psi'=\psi_1$. \newline
	
	The functions $(s-1)(w-1)(s+w-3/2)Z(s,w;\psi,\psi')$ are polynomially bounded in vertical strips, in the sense that for any $C_1>0$, there is a constant $C_2>0$ such that $|(s-1)(w-1)(s+w-3/2)Z(s,w;\psi,\psi')| \ll ((1+\Im(s))(1+\Im(w)))^{C_2}$ whenever $|\Im(s)|, |\Im(w)| \leq C_1$. The functions satisfy the functional equations,
\begin{align} \label{AswAws}
 Z(s,w;\psi,\psi')=Z(w,s;\psi',\psi).
\end{align}
  Additionally, if $\psi \neq \psi_1$, then
\begin{align} \label{Zfcneqnpsinot1}
  Z(1-s,s+w-\half;\psi,\psi')=\pi^{1-2s}N(\psi)^{(2s-1)/2}\frac{\Gamma(s)}{\Gamma(1-s)}Z(s,w;\psi,\psi'),
\end{align}
  where we denote $N(\psi)$ to be $N(q)$ with $q$ being the modulus of $\psi$. \newline

Moreover, if $\psi =\psi_1$, we have
\begin{align}
\label{Zfcneqnpsi1}
\begin{split}
 Z(1-s, & s+w-\half;\psi_1,\psi') \\
&= \frac 12 \pi^{1-2s} \frac{\Gamma(s)}{\Gamma(1-s)}  \Big ( \frac{1-2^{-(1-s)}}{2(1-2^{-s})}Z(s,w;\psi_1,\psi'(\psi_1 +\psi_{i})(\psi_1+ \psi_{1+i})) \\
& \hspace*{1cm} +\frac{1+2^{-(1-s)}}{2(1+2^{-s})}Z(s,w;\psi_1,\psi'(\psi_1 +\psi_{i})(\psi_1- \psi_{1+i})) + 2^{2s-1}Z(s,w;\psi_1, \psi'(\psi_1-\psi_{i}) \Big ).
\end{split}
\end{align}
\end{proposition}
\begin{proof}

   Note first that the functional equations in \eqref{AswAws} follow easily from the quadratic reciprocity law given in \eqref{quadrec}. To derive the other assertions, we write $m=m_0m_1^2$ with $m_0, m_1$ primary and $\mu_{[i]}^2(m_0)=1$ to get that (noting that $\psi'(m_0m_1^2)=\psi'(m_0)$ since the modulus of $\psi'$ divides $4$ and $m_1$ is primary)
\begin{align*}
	Z(s,w;\psi,\psi')=&
\zeta_2(2s+2w-1)\sum_{\substack{m_0\odd \\\mu_{[i]}^2(m_0)=1}}\sum_{\substack{m_1\odd}}\frac{\psi'(m_0)}{N(m_0)^wN(m_1)^{2w}} \sum_{n\odd}\frac{\leg {m_0m^2_1}n\psi(n)}{N(n)^s}.
\end{align*}

Now as
\begin{align*}
   \sum_{n\odd}\frac{\leg {m_0m^2_1}n\psi(n)}{N(n)^s}=\prod_{\substack{\varpi|m_1}}(1-\chi_{m_0}(\varpi)\psi(\varpi){N(\varpi)^{-s}})L_2(s, \chi_{m_0}\psi),
\end{align*}
we deduce that
\begin{align*}
 Z(s, & w;\psi,\psi') \\
 &= \zeta_2(2s+2w-1)\sum_{\substack{m_0\odd \\\mu_{[i]}^2(m_0)=1}}\sum_{\substack{m_1\odd}}\frac{\psi'(m_0)}{N(m_0)^wN(m_1)^{2w}} \prod_{\substack{\varpi|m_1}}(1-\chi_{m_0}(\varpi)\psi(\varpi){N(\varpi)^{-s}})L_2(s, \chi_{m_0}\psi).
\end{align*}

 Direct computation of the Euler factors renders
\begin{align}
\label{Zrelation}
   Z(s,w;\psi,\psi')=\zeta_2(2s+2w-1)\sum_{\substack{m_0\odd \\\mu_{[i]}^2(m_0)=1}}\frac {\psi'(m_0) L_2(s, \chi_{m_0}\psi)\zeta_2(2w)}{N(m_0)^w L_2(s+2w, \chi_{m_0}\psi)}.
\end{align}

Note that if $\psi \neq \psi_1$, $L_2(s, \chi_{m_0}\psi)= L(s, \chi_{m_0}\psi)$ and the functional equation in \eqref{fneqnL} leads to
\begin{align*}
  L_2(1-s, \chi_{m_0}\psi)=(N(m_0)N(\psi))^{(2s-1)/2}\pi^{1-2s}\frac{\Gamma(s)}{\Gamma(1-s)} L_2(s, \chi_{m_0}\psi).
\end{align*}
  Applying the above to \eqref{Zrelation}, the functional equations in \eqref{Zfcneqnpsinot1} follow. \newline

It still remains to consider the case $\psi = \psi_1$.  If $m_0$ is of type $2$, then $L_2(s, \chi_{m_0}\psi)= L(s, \chi_{m_0}\psi)$ and again applying \eqref{fneqnL} gives
\begin{align*}
  L_2(1-s, \chi_{m_0}\psi_1)=(4N(m_0))^{(2s-1)/2}\pi^{(1-2s)}\frac{\Gamma(s)}{\Gamma(1-s)} L_2(s, \chi_{m_0}\psi).
\end{align*}
  If $m_0$ is of type $1$, we have $L_2(s, \chi_{m_0}\psi)= (1-\chi_{1+i}(m_0)2^{-s})L(s, \chi_{m_0}\psi)$ and applying the functional equation  given in \eqref{fneqnL} leads to
\begin{align*}
  L_2(1-s, \chi_{m_0}\psi_1)= \frac{1-\chi_{1+i}(m_0)2^{-(1-s)}}{1-\chi_{1+i}(m_0)2^{-s}} N(m_0)^{(2s-1)/2}\pi^{1-2s}\frac{\Gamma(s)}{\Gamma(1-s)} L_2(s, \chi_{m_0}\psi_1).
\end{align*}
Thus if $\psi = \psi_1$, we obtain
\begin{align*}
\begin{split}
Z(1-s, & s+w-\half;\psi_1,\psi') \\
& =\pi^{1-2s} \frac{\Gamma(s)}{\Gamma(1-s)}\zeta_2(2s+2w-1) \Big ( \sum_{\substack{m_0\odd \\\mu_{[i]}^2(m_0)=1 \\ m_0 \text{ of type 1}}} \frac{(1-\chi_{1+i}(m_0)2^{-(1-s)})}{(1-\chi_{1+i}(m_0)2^{-s})}\frac {\psi'(m_0) L_2(s, \chi_{m_0}\psi_1)\zeta_2(2w)}{N(m_0)^w L_2(s+2w, \chi_{m_0}\psi_1)} \\
& \hspace*{6cm} +\sum_{\substack{m_0\odd \\\mu_{[i]}^2(m_0)=1 \\ m_0 \text{ of type 2}}}2^{2s-1}\frac {\psi'(m_0) L_2(s, \chi_{m_0}\psi_1)\zeta_2(2w)}{N(m_0)^w L_2(s+2w, \chi_{m_0}\psi_1)} \Big ).
\end{split}
\end{align*}

   We apply the character sums $\frac 12 (\psi_1(m_0) \pm \psi_{i}(m_0))$ to detect $m_0$'s type and recast the above as
\begin{align*}
\begin{split}
Z(1-s, & s+w-\half;\psi_1,\psi') \\
& =\frac 12 \pi^{1-2s}\frac{\Gamma(s)}{\Gamma(1-s)}\zeta_2(2s+2w-1) \\
& \hspace*{1cm} \times  \sum_{\substack{m_0\odd \\ \mu_{[i]}^2(m_0)=1 }} \Big( \frac{1-\chi_{1+i}(m_0)2^{-(1-s)}}{1-\chi_{1+i}(m_0)2^{-s}}\frac {\psi'(m_0)(\psi_1(m_0) +\psi_{i}(m_0)) L_2(s, \chi_{m_0}\psi_1)\zeta_2(2w)}{N(m_0)^w L_2(s+2w, \chi_{m_0}\psi_1)} \\
& \hspace*{2cm} +2^{2s-1}\frac {\psi'(m_0)(\psi_1(m_0)-\psi_{i}(m_0)) L_2(s, \chi_{m_0}\psi_1)\zeta_2(2w)}{N(m_0)^w L_2(s+2w, \chi_{m_0}\psi_1)} \Big ).
\end{split}
\end{align*}

Now $\frac 12 (\psi_1(m_0) \pm \psi_{1+i}(m_0))$ can be used to detect $\psi_{1+i}(m_0)=\pm 1$, further recasting the above as
\begin{align*}
\begin{split}
Z(&1-s,  s+w-\half;\psi_1,\psi') \\
 &=\frac 12 \pi^{1-2s} \frac{\Gamma(s)}{\Gamma(1-s)} \zeta_2(2s+2w-1) \\
& \hspace*{1cm} \times \sum_{\substack{m_0\odd \\ \mu_{[i]}^2(m_0)=1 }} \Big ( \frac{1-2^{-(1-s)}}{2(1-2^{-s})}\frac {\psi'(m_0)(\psi_1(m_0) +\psi_{i}(m_0))(\psi_1(m_0) + \psi_{1+i}(m_0)) L_2(s, \chi_{m_0}\psi_1)\zeta_2(2w)}{N(m_0)^w L_2(s+2w, \chi_{m_0}\psi_1)} \\
&\hspace*{1cm} +\frac{1+2^{-(1-s)}}{2(1+2^{-s})}\frac {\psi'(m_0)(\psi_1(m_0) +\psi_{i}(m_0))(\psi_1(m_0) - \psi_{1+i}(m_0)) L_2(s, \chi_{m_0}\psi_1)\zeta_2(2w)}{N(m_0)^w L_2(s+2w, \chi_{m_0}\psi_1)} \\
&\hspace*{1cm} +2^{2s-1}\frac {\psi'(m_0)(\psi_1(m_0)-\psi_{i}(m_0)) L_2(s, \chi_{m_0}\psi_1)\zeta_2(2w)}{N(m_0)^w L_2(s+2w, \chi_{m_0}\psi_1)} \Big ).
\end{split}
\end{align*}

   We now rewrite the right-hand side expression above in terms of $Z(s,w;\psi_1,\psi'')$ for various $\psi''$ to see that the functional equations given in \eqref{Zfcneqnpsi1} follow. \newline

Next Cauchy's inequality and \eqref{L4est} imply that for any $\psi \in \text{CG}$ and any $\Re(s) \geq \frac{1}{2}$, $|s-1|>\varepsilon$,
\begin{align*}
\sumstar_{\substack{(m,2)=1 \\ N(m) \leq X}} |L(s,\chi_{m}\psi)|
\ll X^{1+\varepsilon} |s|^{1/2+\varepsilon}.
\end{align*}
  Applying the functional equation of $L(s, \chi_{m_0}\psi)$ if $\Re(s)<1/2$, we deduce from the above and partial summation that when $\psi \neq \psi_1$, the sum on the right-hand side of \eqref{Zrelation} converges absolutely in the region
$$\{(s,w):\Re(w)>1\text{ and }\Re(s+w)>3/2\}.$$
Now when $\psi=\psi_1$, the summand in \eqref{Zrelation} with $m_0=1$ is $\frac{\zeta_2(s)\zeta_2(2w)}{\zeta_2(s+2w)}$, which implies that there is a pole of $Z(s,w;\psi_1,\psi')$ at $s=1$. Given that
the residue of $\zeta_K(s)$ at $s = 1$ equals $\pi/4$ (see \cite[Theorem 3.8]{iwakow}), it follows that $\res_{(1,w)}Z(s,w;\psi_1, \psi')=\pi \zeta_2(2w)/8$. \newline

As the remaining assertions of the theorem can be established in manners similar to those of \cite[Lemma 2]{Blomer11}, \cite[Theorem 3.1]{Cech} and \cite[Proposition 4.11]{DGH}, we omit here.
\end{proof}

\subsection{Completion of Proof of Theorem \ref{Thmfirstmomentatcentral}}

   It follows from \eqref{AZ} and Proposition \ref{ExtensionOfZ(s,w)} that the right-hand side expression of \eqref{MellinInversion} is holomorphic for $s=1/2$ and $\Re(w)$ large enough. As the left-hand side expression of \eqref{MellinInversion} is meromorphic for all $s$, we conclude that they agree at $s=1/2$, which implies that for large $\omega$,
\begin{align}
\label{Firstmomentintrep}
\sum_{m\odd} L(\tfrac 12, \chi_m)\Phi \left(\frac {N(m)}X \right)=\frac1{2\pi i}\int\limits_{(\omega)}A(\tfrac 12,w)X^w\hat\Phi(w) \dif w=\frac1{2\pi i}\int\limits_{(\omega)}\frac {Z(\frac 12,w)}{\zeta_2(2w)}X^w\hat\Phi(w) \dif w.
\end{align}

  We now shift the last integral above to $\Re(w)=1/2+\varepsilon$ to encounter a double pole at $w=1$ by Proposition \ref{ExtensionOfZ(s,w)}. Calculating the residue then leads to the main term given in \eqref{FirstmomentSmoothed}.  We apply \eqref{h1bound} and Proposition \ref{ExtensionOfZ(s,w)} again to see that the integration on the new line is bounded by the error term given in \eqref{FirstmomentSmoothed} with $\delta=1/2$. \newline

If we assume the truth of GRH, then we can shift the integral in \eqref{Firstmomentintrep} further to the line $\Re(w)=1/4+\varepsilon$.  GRH implies (see \cite[Theorem 5.19]{iwakow}) that for any $\varepsilon>0$,
\begin{align*}
\begin{split}
      \frac1{\zeta_K(s)|}\ll & |s|^{\varepsilon}, \quad \Re(s) \geq 1/2+\varepsilon.
\end{split}
\end{align*}
Applying the above bound, together with \eqref{h1bound}, gives that the integral on the new line is bounded by the error term given in \eqref{FirstmomentSmoothed} with $\delta=1/4$ under GRH. This completes the proof of Theorem \ref{Thmfirstmomentatcentral}.

\section{Proof of Theorem \ref{MainTheoremSmoothed}}

We apply the Mellin inversion to $\sw$ twice.  This gives that for $\Re(s)=\sigma$, $\Re(w)=\omega$ large enough,
\begin{equation*}
\sw=\bfrac1{2\pi i}^2\int\limits_{(\sigma)}\int\limits_{(\omega)}A(s,w)X^wY^s\hat\Phi(w)\hat\Psi(s) \dif w \dif s.
\end{equation*}

Using the relation \eqref{AZ} between $A(s,w)$ and $Z(s,w)$, we recast the above as
\begin{equation}
\label{sw}
\sw=\bfrac1{2\pi i}^2\int\limits_{(2)}\int\limits_{(2)}\frac {Z(s,w)}{\zeta_2(2s+2w-1)}X^wY^s\hat\Phi(w)\hat\Psi(s) \dif w \dif s.
\end{equation}
 Here we set the lines of integration on $\Re(s)=\Re(w)=2$ so that $Z(s, w)$ converges absolutely.  Noticing that the estimation given in \eqref{h1bound} continues to hold with $\hat \Phi$ replaced by $\hat\Psi$. Using this bound for both $\hat \Phi$ and $\hat\Psi$ together with the polynomial bound of $Z(s,w)$ in vertical strips given in Proposition \ref{ExtensionOfZ(s,w)}, we may shift the integral over $w$ in \eqref{sw} to the line $\Re(w)=3/4+\varepsilon$, passing a simple pole at $w=1$. Applying Proposition \ref{ExtensionOfZ(s,w)} for the residue, we thus obtain
\begin{align}
\label{ShiftingToOmega'}
\sw=& \lz \frac1{2\pi i}\pz^2\int\limits_{(2)}\int\limits_{(3/4+\varepsilon)}\frac {Z(s,w)}{\zeta_2(2s+2w-1)}X^wY^s\hat \Phi(w)\hat\Psi(s) \dif w \dif s +\frac{\pi\hat\Phi(1)X}{8}\cdot \frac{1}{2\pi i}\int\limits_{(2)}\frac{Y^s\hat\Psi(s)\zeta_2(2s)}{\zeta_2(2s+1)} \dif s.
\end{align}

  We estimate the last integral in \eqref{ShiftingToOmega'} by further shifting to the line $\Re(s)=\delta$ with $\delta=-\frac14+\varepsilon$ or $\delta=\varepsilon$ if GRH is true.  We encounter a simple pole at $s=1/2$ and after computing the residue, we obtain
\begin{equation}
\label{ComputingContributionOfFirstResidue}
\begin{aligned}
\frac{\pi\hat\Phi(1)X}{8}\cdot \frac{1}{2\pi i}\int\limits_{(2)}\frac{Y^s\hat\Psi(s)\zeta_2(2s)}{\zeta_2(2s+1)}\dif s=\frac{\pi^2\hat\Phi(1)\hat\Psi\bfrac12XY^{1/2}}{128\zeta_2(2)}+\frac{\pi\hat\Phi(1)X}{8}\cdot \frac{1}{2\pi i}\int\limits_{(\delta)}\frac{Y^s\hat\Psi(s)\zeta_2(2s)}{\zeta_2(2s+1)} \dif s.
\end{aligned}
\end{equation}

   To facilitate our estimation of the last integral above and other similar integrals in the sequel, we gather here a few bounds on $\zeta_K(s)$ and $L(s, \chi)$ for any Hecke character $\chi$ modulo $q$ of trivial infinite type. We have
\begin{align}
\label{zetaest}
\begin{split}
  \zeta_K(s) \ll & \left( 1+|s|^2 \right)^{1-\Re(s)/2+\varepsilon}, \quad 0 \leq \Re(s) \leq 1, \\
  \big|L(\tfrac 12+\varepsilon+it, \chi)\big|^{-1} \ll & \big(N(q)(1+|t|)\big)^{\varepsilon}, \quad \text{on GRH}.
\end{split}
\end{align}
  The first inequality in \eqref{zetaest} is the convexity bound for $\zeta_K(s)$ (see \cite[Exercise 3, p. 100]{iwakow}) and the second follows from  \cite[Theorem 5.19]{iwakow}.  The implied constants in both of the estimates depend on $\varepsilon$ alone. \newline

Applying the bounds given in \eqref{zetaest} on the last integral in \eqref{ComputingContributionOfFirstResidue},  we arrive at
\begin{equation*}
\frac{\pi\hat\Phi(1)X}{8}\cdot \frac{1}{2\pi i}\int\limits_{(2)}\frac{Y^s\hat\Psi(s)\zeta_2(2s)}{\zeta_2(2s+1)} \dif s=\frac{\pi^2\hat\Phi(1)\hat\Psi\bfrac12XY^{1/2}}{128\zeta_2(2)}+O\lz XY^{\delta}\pz.
\end{equation*}
This and \eqref{ShiftingToOmega'} render
\begin{align*}
\sw=&\lz \frac1{2\pi i}\pz^2\int\limits_{(2)}\int\limits_{(3/4+\varepsilon)}\frac {Z(s,w)}{\zeta_2(2s+2w-1)}X^wY^s\hat \Phi(w)\hat\Psi(s) \dif w \dif s +\frac{\pi^2\hat\Phi(1)\hat\Psi\bfrac12XY^{1/2}}{128\zeta_2(2)}+O(XY^{\delta}).
\end{align*}

  We now interchange the order of integrations in the double integral above and shift the integral over $s$ to $\Re(s)=3/4$, crossing the polar line at $s=1$.  Computing the residue as before, we see that
\begin{align}
\label{TheoremAfter2Shifts}
\begin{split}
	\sw =& \frac{\pi^2(\hat\Phi(1)\hat\Psi\bfrac12XY^{1/2}+\hat\Psi(1)\hat\Phi\bfrac12YX^{1/2})}{128\zeta_2(2)} \\
& \hspace*{1cm} +\lz \frac1{2\pi i}\pz^2\int\limits_{(3/4)}\int\limits_{(3/4+\varepsilon)}\frac {Z(s,w)}{\zeta_2(2s+2w-1)}X^wY^s\hat \Phi(w)\hat\Psi(s) \dif w \dif s +O_\varepsilon\lz YX^\delta+XY^\delta\pz,
\end{split}
\end{align}
where $\delta=-1/4+\varepsilon$ or $\delta=\varepsilon$ if GRH is true. \newline

  We further move the integral in \eqref{TheoremAfter2Shifts} over $w$ to the left. It follows from Proposition \ref{ExtensionOfZ(s,w)} that  except for the line $s+w=3/2$, the integrand has no poles if $\Re(s+w)\geq1$ or $\Re(s+w)>3/4$ under GRH.  Thus we have, upon setting $\delta'=1/4+\varepsilon$ or $\varepsilon$ if GRH is true,
\begin{align}
\label{ShiftingPast3/2}
\begin{split}
	S_{[i]} (X,Y; & \Phi, \Psi) \\
	=& \frac{\pi^2(\hat\Phi(1)\hat\Psi\bfrac12XY^{1/2}+\hat\Psi(1)\hat\Phi\bfrac12YX^{1/2})}{128\zeta_2(2)} +\lz \frac1{2\pi i}\pz^2\int\limits_{(3/4)}\int\limits_{(\delta')}\frac {Z(s,w)}{\zeta_2(2s+2w-1)}X^wY^s\hat \Phi(w)\hat\Psi(s) \dif w \dif s\\
&\hspace*{1cm} +\frac1{2\pi i}\int\limits_{(3/4)}X^{3/2-s}Y^{s}\hat\Phi\lz\tfrac32-s\pz\hat\Psi\lz s\pz\res_{\lz s,\frac32-s\pz}\frac {Z(s,w)}{\zeta_2(2s+2w-1)} \dif s +O_\varepsilon\lz YX^\delta+XY^\delta\pz, \\
=& \frac{\pi^2(\hat\Phi(1)\hat\Psi\bfrac12XY^{1/2}+\hat\Psi(1)\hat\Phi\bfrac12YX^{1/2})}{128\zeta_2(2)}  \\
&\hspace*{1cm} + \frac1{2\pi i}\int\limits_{(3/4)}X^{3/2-s}Y^{s}\hat\Phi\lz\tfrac32-s\pz\hat\Psi\lz s\pz\res_{\lz s,\frac32-s\pz}\frac {Z(s,w)}{\zeta_2(2s+2w-1)} \dif s +O_\varepsilon\lz YX^\delta+XY^\delta\pz,
\end{split}
\end{align}
  where the last estimation above follows from \eqref{h1bound}, \eqref{zetaest} and the polynomial bound of $Z(s,w)$ in vertical strips. \newline

  Lastly, to compute the residue in the last integral of \eqref{ShiftingPast3/2},  we set $\psi'=\psi_1$ in the functional equation given in \eqref{Zfcneqnpsi1} to see that
\begin{align}
\label{resZ}
\begin{split}
\res_{w=1}&Z(1-s,s+w-\half;\psi_1,\psi_1) \\
=& \frac{\pi^{1-2s}}{2} \frac{\Gamma(s)}{\Gamma(1-s)}\res_{w=1} \Big ( \frac{1-2^{-(1-s)}}{2(1-2^{-s})} Z(s,w;\psi_1,\psi_1(\psi_1 +\psi_{i})(\psi_1+ \psi_{1+i})) \\
& \hspace*{1cm} +\frac{1+2^{-(1-s)}}{2(1+2^{-s})} Z(s,w;\psi_1,\psi_1(\psi_1 +\psi_{i})(\psi_1- \psi_{1+i})) +2^{2s-1}Z(s,w;\psi_1, \psi_1(\psi_1-\psi_{i}) \Big ) \\
 =& \frac{\pi^{1-2s}}{2} \frac{\Gamma(s)}{\Gamma(1-s)} \res_{w=1} \Big ( \frac{1-2^{-(1-s)}}{2(1-2^{-s})}Z(s,w;\psi_1,\psi_1)+\frac{1+2^{-(1-s)}}{2(1+2^{-s})}Z(s,w;\psi_1,\psi_1) +2^{2s-1}Z(s,w;\psi_1, \psi_1) \Big ),
\end{split}
\end{align}
  where the last estimate above follows by noting that $w=1$ is a polar line of $Z(u,v;\psi,\psi')$ if and only if $\psi'=\psi_1$. \newline

  We apply Proposition \ref{ExtensionOfZ(s,w)} to compute the residue appearing on the last expression of \eqref{resZ} to see that
\begin{align*}
\begin{split}
\res_{w=1}Z(1-s,s+w-\half;\psi_1,\psi_1)
 =\frac {\pi^{2-2s}}{16} \zeta_2(2s) \frac{\Gamma(s)}{\Gamma(1-s)}  \Big ( \frac{1-2^{-(1-s)}}{2(1-2^{-s})}+\frac{1+2^{-(1-s)}}{2(1+2^{-s})} +2^{2s-1} \Big ).
\end{split}
\end{align*}

   We further make a change of variable $s \rightarrow 1-s$ to conclude from the above that for all $s\in \mc$,
\begin{align*}
\begin{split}
 \res_{\lz s,\frac32-s\pz} & Z(s,w) \\
=& \frac {\pi^{s}}{16} \zeta_2(2(1-s)) \frac {\Gamma(1-s)}{\Gamma(s)}  \Big ( \frac{1-2^{-s}}{2(1-2^{-(1-s)})} +\frac{1+2^{-s}}{2(1+2^{-(1-s)})} +2^{1-2s} \Big ) \\
=& \frac {\pi^{s}}{16} \zeta_K(2(1-s)) \frac {\Gamma(1-s)}{\Gamma(s)}  \Big ( \frac 12(1+2^{-(1-s)})(1-2^{-s})+\frac 12(1-2^{-(1-s)})(1+2^{-s}) +2^{1-2s}(1-2^{-2(1-s)}) \Big ) \\
=& \frac {\pi^{s}}{16} \zeta_K(2(1-s)) \frac {\Gamma(1-s)}{\Gamma(s)}  \Big ( \frac 12 +2^{1-2s}(1-2^{-2(1-s)}) \Big ) \\
=& \frac {\pi^{s}2^{1-2s}}{16} \zeta_K(2(1-s)) \frac {\Gamma(1-s)}{\Gamma(s)}.
\end{split}
\end{align*}

   We now apply the functional equation in \eqref{fneqnL} to see that
\begin{align*}
  \zeta_K(2(1-s)) =\pi^{(1-2(2s-1))} \frac{\Gamma(2s-1)}{\Gamma(1-(2s-1))} \zeta_K(2s-1)=\pi^{3-4s} \frac{\Gamma(2s-1)}{\Gamma(2-2s)}\zeta_K(2s-1) .
\end{align*}

We then conclude that
\begin{align*}
\begin{split}
 &\res_{\lz s,\frac32-s\pz}Z(s,w) =\pi^{3-3s}2^{-3-2s}\frac {\Gamma(1-s)\Gamma(2s-1)}{\Gamma(2-2s)\Gamma(s)}\zeta_K(2s-1).
\end{split}
\end{align*}

   It follows that the last integral on the right-hand side expression of \eqref{ShiftingPast3/2} is
\begin{align*}
\begin{split}
	\frac{X^{3/2}}{\zeta_2(2)}\int\limits_{(3/4)}\bfrac{Y}{X}^s\hat\Phi\lz\tfrac32-s\pz\hat\Psi(s)\pi^{3-3s}2^{-3-2s}\frac {\Gamma(1-s)\Gamma(2s-1)}{\Gamma(2-2s)\Gamma(s)}\zeta_K(2s-1) \dif s.
\end{split}
\end{align*}

	Combining the above with  \eqref{ShiftingPast3/2} now allows us to complete the proof of Theorem \ref{MainTheoremSmoothed}.

\vspace*{.5cm}

\noindent{\bf Acknowledgments.}   P. G. is supported in part by NSFC Grant 11871082 and L. Z. by the Faculty Silverstar Grant PS65447 at the University of New South Wales (UNSW).  The authors thank the anonymous referee for his/her careful reading of the manuscript.

\bibliography{biblio}
\bibliographystyle{amsxport}

\vspace*{.5cm}

\noindent\begin{tabular}{p{6cm}p{6cm}p{6cm}}
School of Mathematical Sciences & School of Mathematics and Statistics \\
Beihang University & University of New South Wales \\
Beijing 100191 China & Sydney NSW 2052 Australia \\
Email: {\tt penggao@buaa.edu.cn} & Email: {\tt l.zhao@unsw.edu.au} \\
\end{tabular}

\end{document}